\RequirePackage{fix-cm}
\documentclass[smallextended]{svjour3}       
\smartqed  
\usepackage{amsfonts,epsf,amsmath,amssymb,graphicx}
\usepackage{epstopdf}
\usepackage[top=3.5cm,bottom=3.5cm,left=3.5cm,right=3.5cm]{geometry}
\begin{document}


\title{Formula for calculating the Wiener polarity index with applications to benzenoid graphs and phenylenes
}


\author{Niko Tratnik  
}

%
\institute{N. Tratnik \at
              Faculty of Natural Sciences and Mathematics, University of Maribor, Slovenia \\
              \email{niko.tratnik@um.si, niko.tratnik@gmail.com}           
}

\date{Received: \today / Accepted: date}

\maketitle

\begin{abstract}
The Wiener polarity index of a graph is defined as the number of unordered pairs of vertices at distance three. In recent years, this topological index was extensively studied since it has many known applications in chemistry and also in network theory. In this paper, we generalize the result of Behmaram, Yousefi-Azari, and Ashrafi proved in 2012 for calculating the Wiener polarity index of a graph. An important advantage of our generalization is that it can be used for graphs that contain $4$-cycles and also for graphs whose different cycles have more than one common edge.
In addition, using the main result a closed formula for the Wiener polarity index is derived for phenylenes and recalculated for catacondensed benzenoid graphs. The catacondensed benzenoid graphs and phenylenes attaining the extremal values with respect to the Wiener polarity index are also characterized.
\end{abstract}

\section{Introduction and preliminaries}

Throughout the paper, all the considered graphs are simple, finite, and connected. For a graph $G$, $V(G)$ and $E(G)$ denote the sets of all the vertices and edges, respectively. The {\em distance} $d_G(x,y)$ between vertices $x$ and $y$ of a graph $G$ is the length of a shortest path between vertices $x$ and $y$ in $G$. We sometimes write $d(x,y)$ instead of $d_G(x,y)$. For a vertex $x$, we will denote by $N(x)$ the set of all the vertices in $G$ that are adjacent to $x$. The \textit{degree} of $x$, denoted as $\deg_G(x)$ or shortly $\deg(x)$, is the cardinality of the set $N(x)$. For any $k \geq 3$, a cycle on $k$ vertices will be called a \textit{$k$-cycle}. Moreover, the set of all $k$-cycles in $G$ is denoted as $C_k(G)$. A cycle or a path in graph $G$ will be often denoted as the sequence of its vertices. In addition, a \textit{star graph} $S_k$, $k \geq 1$, is a complete bipartite graph $K_{1,k}$.

The Wiener polarity index is a molecular descriptor defined as the number of unordered pairs of vertices at distance three. More precisely, for a graph $G$, the \textit{Wiener polarity index} of $G$, denoted by $W_p(G)$, is defined as
$$W_p(G) = | \lbrace \lbrace u,v \rbrace \subseteq V(G) \, | \, d(u,v)=3 \rbrace |.$$
The Wiener polarity index was introduced in \cite{Wiener}, where H.\,Wiener used this index and the classical Wiener index to predict boiling points of paraffins. Later, Lukovits and Linert \cite{luko} demonstrated quantitative structure-property relationships for the Wiener polarity index in a series of acyclic and cycle-containing hydrocarbons. Furthermore, Hosoya and Gao analysed very good correlation of this index with liquid density \cite{hosoya}. In the same work, remarkably good correlations with refractive index and molar volume are also mentioned. In recent years, the Wiener polarity index was mostly studied on trees \cite{ali,ashrafi,deng,deng1,lei,liu2} and also on unicyclic graphs \cite{hou,liu}. Moreover, the Nordhaus-Gaddum-type results for this index were considered \cite{hua,zhang} and the index was used to study properties of networks \cite{lei1}. Also, in \cite{tratnik} a cut method for the Wiener polarity index was developed for nanotubes and benzenoid graphs that can be embedded into the regular hexagonal lattice.  Some other recent investigations on the Wiener polarity index can be found in \cite{hua1,ilic,il-il}. 

The Zagreb indices have been introduced in 1972 by Gutman and Trinajsti\' c \cite{gut_tri}.  For a graph $G$ with at least one edge, the \textit{first Zagreb index}, $M_1(G)$, and the \textit{second Zagreb index}, $M_2(G)$, are defined as
\begin{eqnarray*}
M_1(G) & = & \sum_{v \in V(G)} \deg(v)^2, \\
M_2(G) & = & \sum_{e=uv \in E(G)} \deg(u) \cdot \deg(v).
\end{eqnarray*}
If $G$ has no edges, we set $M_1(G)=M_2(G)=0$. Since the introduction, many mathematical and chemical properties were found for these topological indices. For example, it is easy to check that 
\begin{equation} \label{zag1}
M_1(G) = \sum_{e=uv \in E(G)} (\deg(u) + \deg(v)).
\end{equation}

In \cite{beh} the Wiener polarity index was expressed by using the Zagreb indices and the following theorem was proved.

\begin{theorem} \cite{beh} \label{behas}
Suppose $G$ is a connected triangle- and quadrangle-free graph such that its different cycles have at most one common
edge. If $N_p = N_p(G)$ and $N_h = N_h(G)$ denote the number of pentagons and hexagons of $G$, then $W_p(G) = M_2(G) - M_1(G) -
5N_p - 3N_h + |E(G)|$.
\end{theorem}

However, many chemical graphs contain $4$-cycles and consequently, Theorem \ref{behas} can not be used in such cases. Therefore, in Theorem \ref{glavni} we generalize this result such that it can be used also for graphs that contain $4$-cycles. Moreover, in our generalization the number of common edges of different cycles is limited only for small cycles.

To show the application of the main result, in the final section we derive a closed formula for the Wiener polarity index of phenylenes, which are important and extensively studied molecular graphs that contain $4$-cycles. For catacondensed benzenoid graphs similar formula was obtained in \cite{beh}, but it does not coincide with our computations. Therefore, we compute the Wiener polarity index also for these graphs. Moreover, the graphs attaining the extremal values with respect to the Wiener polarity index are characterized.

\section{Formula for calculating the Wiener polarity index}

In this section we prove the main result of the paper and it enables us to easily compute the Wiener polarity index for a large class of graphs. We start with two definitions.

\begin{definition} Let $G$ be a graph and $C$ a $4$-cycle in $G$. An edge $e$ \textbf{exits} $C$ if $e$ and $C$ have exactly one common vertex. See Figure \ref{exit}. 
\end{definition}

\begin{figure}[h!] 
\begin{center}
\includegraphics[scale=0.7]{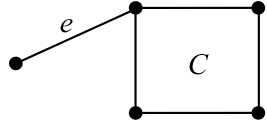}
\end{center}
\caption{\label{exit} A cycle $C$ and an edge $e$ such that $e$ exits $C$.}
\end{figure}

\noindent
Obviously, if $G$ does not contain $3$-cycles, then for a given $4$-cycle $C=u_1,u_2,u_3,u_4,u_1$ the number of edges exiting $C$ is exactly $\deg(u_1) + \deg(u_2)+  \deg(u_3)+ \deg(u_4) - 8$.
\begin{definition}
Let $G$ be a graph without 3-cycles. If $G$ has at least one $4$-cycle, we define

$$f(G) = \sum_{C=u_1,u_2,u_3,u_4,u_1 \in C_4(G)} (\deg(u_1) + \deg(u_2)+ \deg(u_3)+ \deg(u_4) - 8).$$
Otherwise, $f(G)=0.$
\end{definition}

\noindent
It is easy to notice that $f(G)$ is exactly the number of ordered pairs $(C,e)$, where $C$ is a cycle of $G$ and $e$ exits $C$. Now we are able to state the main result of the paper. In the proof, vertices $u,v$ of a 6-cycle $C$ are called \textit{diametrically opposite} if $d_C(u,v)=3$.

\begin{theorem}
\label{glavni}
Let $G$ be a connected graph without $3$-cycles. Moreover, suppose that any two distinct cycles $C_1,C_2 \in C_4(G) \cup C_5(G) \cup C_6(G)$ have at most two common edges and any two cycles $C_1,C_2 \in C_4(G)$ have at most one common edge. Then it holds
$$W_p(G) = M_2(G) - M_1(G) - f(G) - 4|C_4(G)| - 5|C_5(G)| - 3|C_6(G)| + |E(G)|.$$
\end{theorem}

\begin{proof}
If $G$ is a star graph or a graph with only one vertex, it can be easily checked that the theorem holds. Therefore, in the rest of the proof we assume that $G$ has more than one vertex and that $G$ is not a star graph. Let $e=xy \in E(G)$ be an edge with $\deg(x) > 1$ and $\deg(y)>1$. Moreover, let $u,v \in V(G) \setminus \lbrace x,y \rbrace$ be two distinct vertices such that $ux \in E(G)$ and $vy \in E(G)$. Obviously, $P=u,x,y,v$ is a path of length three in $G$, see Figure \ref{potP}.

\begin{figure}[h!] 
\begin{center}
\includegraphics[scale=0.7]{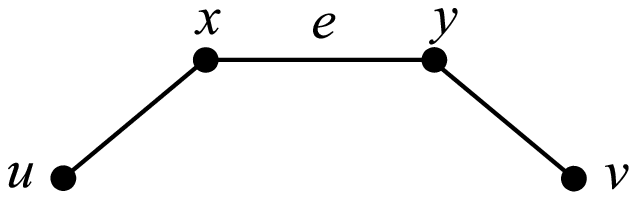}
\end{center}
\caption{\label{potP} Path $P$.}
\end{figure}

We define
$$p_3(G) = \sum_{e=xy \in E(G)}(\deg(x)-1)(\deg(y)-1).$$
\end{proof}

\noindent
One can notice that $p_3(G)$ is exactly the number of all the paths of length three in $G$. By a simple calculation we obtain

\begin{eqnarray*}
p_3(G) & = & \sum_{e=xy \in E(G)} \Big[ \deg(x)\deg(y) - (\deg(x) + \deg(y))+ 1) \Big] \\
 & = & \sum_{e=xy \in E(G)}\deg(x)\deg(y) - \sum_{e=xy \in E(G)}(\deg(x) + \deg(y)) + \sum_{e=xy \in E(G)} 1 \\
 & = & M_2(G) - M_1(G) + |E(G)|,
\end{eqnarray*}
where the last equality follows by Equation \eqref{zag1}.

However, for two vertices $u,v \in V(G)$ there may exists more than one path of length three between them. Suppose that there is some other path (distinct from $P$) $P'=u,a,b,v$ of length three between $u$ and $v$. Consider the following options:
\begin{itemize}
\item [(a)] $x \notin V(P')$ and $y \notin V(P')$: \\
In this case, $u$ and $v$ lie on a $6$-cycle $x,u,a,b,v,y,x$, see Figure \ref{sita}. Moreover, this is the only $6$-cycle in $G$ containing $u$ and $v$ as diametrically opposite vertices (otherwise we have two distinct $6$-cycles with three edges in common).

\begin{figure}[h!] 
\begin{center}
\includegraphics[scale=0.7]{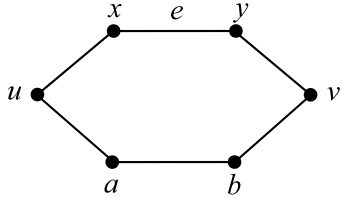}
\end{center}
\caption{\label{sita} Paths $P$ and $P'$ is case (a).}
\end{figure}

\item [(b)] $x=a$ and $y \notin V(P')$: \\
In this case, we obtain a $4$-cycle $C=x,b,v,y,x$ such that the edge $xu$ exits $C$, see Figure \ref{sitb}. Moreover, $P'$ is the only such path (otherwise we obtain two distinct $4$-cycles with two edges in common).

\begin{figure}[h!] 
\begin{center}
\includegraphics[scale=0.7]{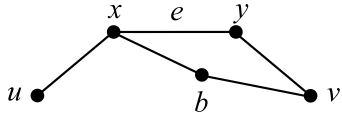}
\end{center}
\caption{\label{sitb} Paths $P$ and $P'$ in case (b).}
\end{figure}

 \item [(c)] $y=b$ and $x \notin V(P')$: \\
In this case, we obtain a $4$-cycle $C=u,x,y,a,u$ such that  the edge $yv$ exits $C$. Moreover, $P'$ is the only such path (otherwise we obtain two distinct $4$-cycles with two edges in common).
\end{itemize}

\noindent
Next we show that cases (a) and (b) or (a) and (c) can not happen simultaneously. Suppose that there is a path $P'=u,x,b,v$ of length $3$ between $u$ and $v$ such that $y\neq b$ (case (b)). Moreover, let $P''=u,a',b',v$ be another path between $u$ and $v$ such that $x \notin V(P'')$ and $y \notin V(P'')$ (case (a)).
\begin{itemize}
\item If $b'=b$, then we obtain a $4$-cycle $x,y,v,b,x$ and a $6$-cycle $u,x,y,v,b,a',u$ with three common edges - a contradiction.
\item If $b'\neq b$, then we obtain a $6$-cycle $u,a',b',v,y,x,u$ and a $6$-cycle $u,a',b',v,b,x',u$ with three common edges - a contradiction.
\end{itemize}

\noindent
Since we get a contradiction in every case, (a) and (b) can not both happen. In a similar way we show that (a) and (c) can not both happen. However, cases (b) and (c) can both happen.

On the other hand, any $6$-cycle contains exactly three pairs of diametrically opposite vertices and every $4$-cycle $u_1,u_2,u_3,u_4,u_1$ has exactly $\deg(u_1) + \deg(u_2)+\deg(u_3)+\deg(u_4) - 8$ exiting edges. Therefore,
$$p_3(G) - f(G) - 3|C_6(G)|$$
represents the number of unordered pairs of distinct vertices $u,v$ for which there exists a path of length three between them.

However, it can happen that $d(u,v) < 3$. In such a case we consider two possibilities:
\begin{itemize}
\item [$(i)$] $d(u,v)=1$: \\
In this case, $u$ and $v$ lie on a $4$-cycle $C=x,y,v,u,x$, see Figure \ref{siti}. Obviously, it holds $d_C(u,v)=1$ and $C$ is the only $4$-cycle with such property (otherwise we obtain $4$-cycles with two common edges or a 4-cycle and a $6$-cycle with three common edges).

\begin{figure}[h!] 
\begin{center}
\includegraphics[scale=0.7]{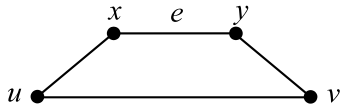}
\end{center}
\caption{\label{siti} $4$-cycle in case $(i)$.}
\end{figure}

\item [$(ii)$] $d(u,v)=2$: \\
In this case, $u$ and $v$ lie on a $5$-cycle $C=x,y,v,a,u,x$ such that $a \in V(G) \setminus \lbrace x,y,u,v \rbrace$, see Figure \ref{sitii}. Obviously, it holds $d_C(u,v)=2$ and $C$ it the only $5$-cycle with such property (otherwise we obtain 5-cycles with three common edges or a 5-cycle and a 6-cycle with three common edges).

\begin{figure}[h!] 
\begin{center}
\includegraphics[scale=0.7]{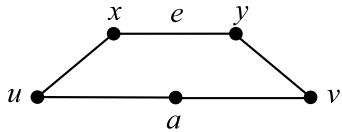}
\end{center}
\caption{\label{sitii} $5$-cycle in case $(ii)$.}
\end{figure}

\end{itemize}

On the other hand, for every $4$-cycle there are exactly four pairs of vertices at distance one and for every $5$-cycle there are exactly five pairs of vertices at distance two. Therefore,
$$p_3(G) - f(G) - 3|C_6(G)| - 4|C_4(G)| - 5|C_5(G)|$$
represents the number of unordered pairs of vertices at distance three, which is exactly the Wiener polarity index of $G$. This completes the proof. \qed

\section{Catacondensed benzenoid graphs and phenylenes}

In this section we apply the main result of the paper to find explicit formulas for the Wiener polarity index of catacondensed benzenoid graphs and phenylenes. First, we need to introduce some additional definitions and notation.

In the existing (both mathematical and chemical) literature, there is inconsistency in the terminology pertaining to (what we call
here) ``benzenoid graph". In this paper, a \textit{benzenoid graph} is a 2-connected plane graph in which the boundary of any inner face is a 6-cycle (and any 6-cycle is the boundary of some inner face), such that two
6-cycles are either disjoint or have exactly one common edge, and no three 6-cycles share a common edge. 

However, in some literature it is assumed that a benzenoid graph can be embedded into the regular hexagonal lattice \cite{gucy-89}. Let $G$ be a benzenoid graph. Any inner face of $G$ is called a \textit{hexagon} of $G$ and the number of all the hexagons of $G$ is denoted by $h(G)$. A vertex shared by three hexagons of $G$ is called an \textit{internal} vertex of $G$. The number of all internal vertices of $G$ will be denoted by $n_i(G)$. A benzenoid graph is said to be \textit{catacondensed} if it does not possess internal vertices. Otherwise it is called \textit{pericondensed}.

We say that two faces of a plane graph are \textit{adjacent} if they have at least one edge in common. The \textit{dualist graph} of a given benzenoid graph $G$ consists of vertices corresponding to hexagons of $G$; two vertices are adjacent if and only if the corresponding hexagons are adjacent. Obviously, the dualist graph of $G$ is a tree if and only if $G$ is catacondensed. For a catacondensed benzenoid graph $G$, its dualistic tree has $h(G)$ vertices and none of its vertices have degree greater than three, see Figure \ref{benzenoid}.

\begin{figure}[h!] 
\begin{center}
\includegraphics[scale=0.6]{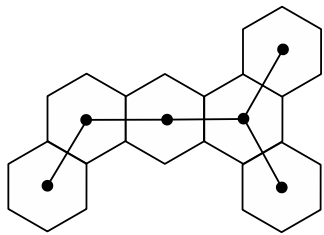}
\end{center}
\caption{\label{benzenoid} A catacondensed benzenoid graph $G$ with its dualistic tree.}
\end{figure}

Let $G$ be a catacondensed benzenoid graph with at least two hexagons. If we add \textit{quadrilaterals} (faces whose boundary is a $4$-cycle) between all pairs of adjacent hexagons of $G$, the obtained graph $G'$ is called a \textit{phenylene}, see Figure \ref{phenylene}. We then say that $G$ is the \textit{hexagonal squeeze} of $G'$. Obviously, graph $G$ in Figure \ref{benzenoid} is the hexagonal squeeze of phenylene $G'$ from Figure \ref{phenylene}. For a phenylene $G'$ the \textit{dualist tree} of $G'$ is defined as the dualist tree of its hexagonal squeeze. Moreover, a \textit{hexagon} of $G'$ is any inner face of $G'$ whose boundary is a $6$-cycle.

\begin{figure}[h!] 
\begin{center}
\includegraphics[scale=0.6]{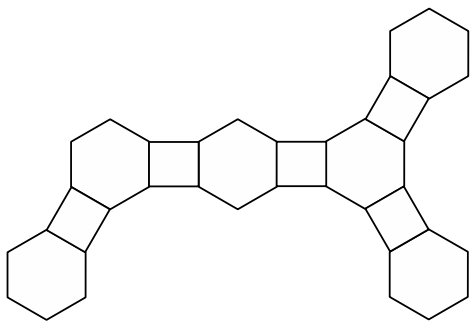}
\end{center}
\caption{\label{phenylene} A phenylene $G'$.}
\end{figure}

Let $G$ be a catacondensed benzenoid graph or a phenylene and let $h_0$ be a hexagon of $G$. Hexagon $h_0$ is called \textit{terminal} if it is adjacent to exactly one other inner face of $G$ and it is called \textit{branched} if it is adjacent to exactly three other inner faces of $G$. If $h_0$ is adjacent to exactly two other inner faces, it has two vertices of degree two. Then $h_0$ is called \textit{angular} if these two vertices are adjacent and \textit{linear} it they are not adjacent. The number of terminal, branched, angular, and linear hexagons of $G$ will be denoted by $t(G)$, $b(G)$, $a(G)$, and $l(G)$, respectively. Moreover, a catacondensed benzenoid graph or a phenylene is called a \textit{linear chain} if it does not contain angular or branched hexagons. A \textit{segment} in a catacondensed benzenoid graph or a phenylene $G$ is any maximal linear chain which is a subgraph of $G$. Finally, the number of all the segments of $G$ will be denoted by $s(G)$. Obviously, for the graph from Figure \ref{benzenoid} it holds $t(G)=3$, $b(G)=a(G)=l(G)=1$, $s(G)= 4$ and the same is true for the graph from Figure \ref{phenylene}. 

For any $k \geq 0$, the number of vertices of a graph $G$ which have degree $k$ will be denoted by $n_k(G)$. We can now state the following proposition.

\begin{proposition} \cite{gut}
\label{gut-ben}
If $G$ is a benzenoid graph, then $|V(G)|  =  4h(G) + 2 - n_i(G)$, 
$ |E(G)|  =  5h(G) + 1 - n_i(G)$, $n_3(G)  =  2h(G) - 2$, and $n_2(G)  =  2h(G) + 4 - n_i(G)$.

\end{proposition}

\begin{remark}
In \cite{gut} Proposition \ref{gut-ben} was proved just for benzenoid  graphs that can be embedded into the regular hexagonal lattice. However, the same proof works for any benzenoid graph.
\end{remark}

\noindent
Using Proposition \ref{gut-ben} it is possible to derive similar formulas also for phenylenes.

\begin{proposition} 
\label{phen}
If $G$ is a phenylene, then $|V(G)|  =  6h(G)$, $|E(G)|  =  8h(G) - 2$, $n_3(G) =  4h(G) - 4$, and $n_2(G)  =  2h(G) + 4$.

\end{proposition}

\begin{proof}
Let $G'$ be the hexagonal squeeze of $G$. Since $h(G)=h(G')$ and $G$ has exactly $h(G)-1$ quadrilaterals, $|V(G)| = |V(G')| + 2(h(G)-1)$ and $|E(G)|= |E(G')| + 3(h(G)-1)$. Moreover, $G$ and $G'$ have the same number of vertices of degree two. Therefore, the result follows by Proposition \ref{gut-ben}. \qed
\end{proof}

\noindent
In the next lemma, we derive relations between the numbers of different types of hexagons and the number of segments.

\begin{lemma} \label{pomoc}
If $G$ is a catacondensed benzenoid graph or a phenylene, then it holds $2b(G) + a(G) = s(G) -1$ and $b(G) + 2 = t(G)$.
\end{lemma}
\begin{proof}
Let $T$ be a dualistic tree of $G$. For any $v \in V(T)$ with $\deg(v)=2$ and $N(v)=\lbrace a, b \rbrace$, let $T_v$ be a tree obtained from $T$ by removing $v$, $va$, $vb$ and adding an edge $ab$. In such a case, we say that $T_v$ is obtained from $T$ by \textit{skipping} vertex $v$. Moreover, let $T'$ be a tree obtained from $T$ by successively skipping all the vertices corresponding to the linear hexagons of $G$. Obviously, $|V(T')| = b(G) + a(G) + t(G)$ and $|E(T')| = s(G)$. Since $T'$ is a tree, we obtain
\begin{equation}
\label{en1}
b(G) + a(G) + t(G) = s(G) +1.
\end{equation}
Moreover, since the sum of all the degrees of vertices in $T'$ is equal to $2|E(T')|$ (by the handshaking lemma) one can conclude
\begin{equation}
\label{en2}
3b(G) + 2a(G) + t(G) = 2s(G).
\end{equation}

\noindent
Finally, the results follow from \eqref{en1} and \eqref{en2}. \qed
\end{proof}

\noindent 
In the next lemma, formulas for the first Zagreb index are given.
\begin{lemma} \label{firzag1}
For a catacondensed benzenoid graph $G_1$ and a phenylene $G_2$ it holds $M_1(G_1)= 26h(G_1)-2$ and $M_1(G_2)=44h(G_2)-20 $.
\end{lemma}


\begin{proof}
The proof for catacondensed benzenoid graph $G_1$ can be found in \cite{beh,doslic}. By Proposition \ref{phen}, $G_2$ has $4h(G_2)-4$ vertices of degree three and $2h(G_2)+4$ vertices of degree two. Therefore,
$$M_1(G_2) = 9(4h(G_2)-4) + 4(2h(G_2)+4) = 44h(G_2) - 20$$
and the proof is complete. \qed
\end{proof}

\noindent 
In the following proposition we obtain formulas for the second Zagreb index with respect to the number of hexagons, the number of segments, and the number of branched hexagons. Note that for catacondensed benzenoid graphs a similar formula was obtained in \cite{doslic}. In addition, for benzenoid graphs this index can be expressed also by using inlets (for the details see Proposition 5.5 in \cite{gut2}).

\begin{proposition} \label{second_zagreb}
If $G_1$ is a catacondensed benzenoid graph and $G_2$ a phenylene, then it holds
\begin{eqnarray*}
M_2(G_1) & = & 33h(G_1) + s(G_1) + b(G_1) - 10, \\
M_2(G_2) & = & 60h(G_2) + s(G_2) + b(G_2) - 37.
\end{eqnarray*}
\end{proposition}

\begin{proof}
In this proof, an edge $e=xy$ will be called an edge of \textit{type 1} if $x$ and $y$ both have degree two. Moreover, if one of $x$ and $y$ has degree two and the other one has degree three, we say that $e$ is of \textit{type 2}. Finally, $e$ is an edge of \textit{type 3} if $x$ and $y$ both have degree three.  

We notice that a terminal hexagon of $G_1$ or $G_2$ has three edges of type 1, two edges of type 2, and one edge of type 3. A linear hexagon has four edges of type 2 and two edges of type 3. An angular hexagon has one edge of type 1, two edges of type 2, and three edges of type 3. Finally, a branched hexagon has six edges of type 3.

If $G_1$ has only one hexagon, the result obviously holds. Otherwise, in a catacondensed benzenoid graph $G_1$ there are $h(G_1)-1$ edges of type 3 that appear on two different hexagons. Therefore, by counting contributions of hexagons and contributions of edges that belong to two hexagons, we get
\begin{eqnarray*}
M_2(G_1) & = & 24t(G_1) + 24l(G_1) + 25a(G_1) + 27b(G_1) + 9(h(G_1)-1) \\
& = & 24h(G_1) + a(G_1) + 3b(G_1) + 9h(G_1) - 9 \\
& = & 33h(G_1) + s(G_1) + b(G_1) - 10,
\end{eqnarray*}
where the last equality follows from Lemma \ref{pomoc}.

In a phenylene $G_2$ there are $4(h(G_1)-1)$ edges of type 3 that appear on a quadrilateral. Therefore, by counting contributions of hexagons and quadrilaterals we get
\begin{eqnarray*}
M_2(G_2) & = & 24t(G_2) + 24l(G_2) + 25a(G_2) + 27b(G_2) + 36(h(G_2)-1) 
\\ & = & 60h(G_2) + s(G_2) + b(G_2) - 37,
\end{eqnarray*}
where the computation is done in a similar way as before. This completes the proof. \qed
\end{proof}

\noindent
The following theorem is the main result of this section.

\begin{theorem} \label{Wie_pol}
If $G_1$ is a catacondensed benzenoid graph and $G_2$ a phenylene, then it holds
\begin{eqnarray*}
W_p(G_1) & = & 9h(G_1) + s(G_1) + b(G_1) - 7, \\
W_p(G_2) & = & 13h(G_2) + s(G_2) + b(G_2) - 11.
\end{eqnarray*}
\end{theorem}

\begin{proof}
Since $G_1$ does not contain $4$-cycles and $5$-cycles, we have $f(G_1)=|C_4(G_1)|=|C_5(G_1)|=0$ and $|C_6(G_1)| = h(G_1)$. On the other hand, for $G_2$ we have $|C_4(G_2)|=h(G_2)-1$, $|C_5(G_2)|=0$, $f(G_2) = 4h(G_2)-4$, and $|C_6(G_2)|=h(G_2)$. Also, it is easy to check that benzenoid graphs and phenylenes suffice the conditions of Theorem \ref{glavni}. Therefore, the desired results follow by Theorem \ref{glavni}, Proposition \ref{gut-ben}, Proposition \ref{phen}, Lemma \ref{firzag1}, and Proposition \ref{second_zagreb}. \qed
\end{proof}

To show one example, let $G$ be a benzenoid graph from Figure \ref{benzenoid} and $G'$ a phenylene from Figure \ref{phenylene}. As already noticed, $h(G)=h(G')=6$, $b(G)=b(G')=1$, and $s(G)=s(G')= 4$. Therefore, by Theorem \ref{Wie_pol} we compute $W_p(G) = 52$ and $W_p(G') = 72$.
\smallskip

In the rest of the section we will characterize the catacondensed benzenoid graphs and phenylenes with exactly $h$ hexagons for which the Wiener polarity index attains the minimum and the maximum value.  Note that the catacondensed benzenoid graphs attaining the lower bound coincide with the graphs obtained in \cite{beh}. However, the result for the maximum value requires some additional insights.

\begin{proposition} 
If $G$ is a catacondensed benzenoid graph or a phenylene with exactly $h$ hexagons, we denote by $L_h$ the linear benzenoid chain or the linear phenylene chain, respectively, that also contains $h$ hexagons. Then $W_p(G) \geq W_p(L_h)$ and the equality holds if and only if $G \cong L_h$.
\end{proposition}

\begin{proof}
By Theorem \ref{Wie_pol}, the Wiener polarity index attains its minimum value if and only if $s(G)=1$ and $b(G)=0$, which is true if and only if $G$ is a linear chain. \qed
\end{proof}

To find the graphs that attain the maximum value, we introduce special families of catacondensed benzenoid graphs and phenylenes. These families will be defined by the following extensions.

\begin{itemize}
\item [$(i)$] \textbf{Extension 1:} \textit{Let $G_0$ be a catacondensed benzenoid graph (or a phenylene) and let $h_0$ be an arbitrary terminal hexagon of $G_0$. We denote by $G$ a graph obtained from $G_0$ by attaching exactly two new hexagons to $h_0$ such that $h_0$ becomes a branched hexagon (and in the case of phenylenes, we also add two quadrilaterals between $h_0$ and the two new hexagons). See Figure \ref{dodas_dva}.}
\begin{figure}[h!] 
\begin{center}
\includegraphics[scale=0.6]{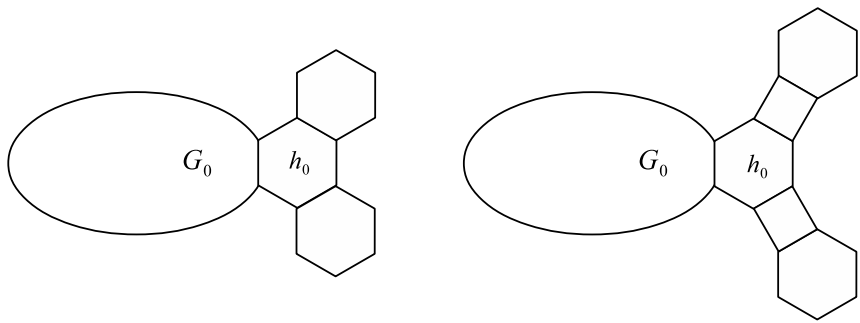}
\end{center}
\caption{\label{dodas_dva} A catacondensed benzenoid graph (or a phenylene) $G$ obtained from $G_0$ by Extension 1.}
\end{figure}
\item [$(ii)$]  \textbf{Extension 2:} \textit{Let $G_0$ be a catacondensed benzenoid graph (or a phenylene) and let $h_0$ be an arbitrary terminal hexagon of $G_0$. We denote by $G$ a graph obtained from $G_0$ by attaching exactly one new hexagon to $h_0$ such that $h_0$ becomes an angular hexagon (and in the case of phenylenes, we also add a quadrilateral between $h_0$ and the new hexagon). See Figure \ref{dodas_enega}.}
\begin{figure}[h!] 
\begin{center}
\includegraphics[scale=0.6]{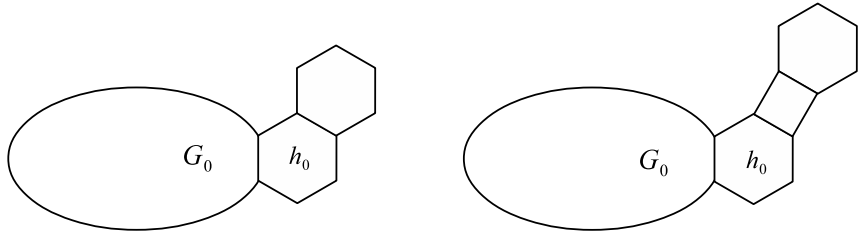}
\end{center}
\caption{\label{dodas_enega} A catacondensed benzenoid graph (or a phenylene) $G$ obtained from $G_0$ by Extension 2.}
\end{figure}

\end{itemize}

\noindent
For any $h \geq 2$, we define families ${\mathcal{B}}_h$ and ${\mathcal{P}}_h$ as follows.

\begin{itemize}
\item [$(i)$] If $h$ is an even number, then ${\mathcal{B}}_h$ (or ${\mathcal{P}}_h$) is the set of all the graphs that can be obtained from the catacondensed benzenoid graph with exactly two hexagons (or from the phenylene with exactly two hexagons) by performing Extension 1 exactly $\left(\frac{h}{2} - 1\right)$ times.

\item [$(ii)$] If $h$ is an odd number, then ${\mathcal{B}}_h$ (or ${\mathcal{P}}_h$) is the set of all the graphs that can be obtained from the catacondensed benzenoid graph with exactly two hexagons (or from the phenylene with exactly two hexagons) by performing Extension 1 exactly $\left(\frac{h-1}{2}-1\right)$ times and performing Extension 2 exactly once (in an arbitrary order).
\end{itemize}

\noindent
For an example, Figure \ref{graf_iz_druzine} shows a catacondensed benzenoid graph with nine hexagons that belongs to ${\mathcal{B}}_9$. Obviously, all the graphs in ${\mathcal{B}}_h \cup {\mathcal{P}}_h$, $h\geq 2$, have exactly $h$ hexagons.
\begin{figure}[h!] 
\begin{center}
\includegraphics[scale=0.6]{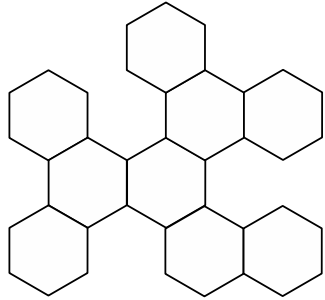}
\end{center}
\caption{\label{graf_iz_druzine} A graph from ${\mathcal{B}}_9$.}
\end{figure}

The next lemma gives an upper bound on the number of segments.

\begin{lemma} \label{st_segmentov}
For a catacondensed benzenoid graph or a phenylene $G$ with $h(G) \geq 2$ it holds $s(G) \leq h(G)-1$. Moreover, for all the graphs in ${\mathcal{B}}_h \cup{\mathcal{P}}_h$, $h \geq 2$, the equality holds. 
\end{lemma}

\begin{proof}
If $G$ has exactly two hexagons, then $s(G)=1$ and the statement holds. Whenever we add exactly one hexagon to a graph, the number of segments stays the same or increases by one. Therefore, by induction we obtain $s(G) \leq h(G)-1$ for any $G$.

From the definition of the families ${\mathcal{B}}_h$ and ${\mathcal{P}}_h$ it follows that whenever we add two hexagons to a smaller graph, the number of segments increases by two, and whenever we add one hexagon, the number of segments increases by one. Therefore, the equality holds for these graphs. \qed
\end{proof}

\noindent
Finally, everything is prepared for the following theorem.
\begin{theorem} 
If $G$ is a catacondensed benzenoid graph (or a phenylene) with $h$ hexagons, $h \geq 2$, and if $G'$ is an arbitrary graph from ${\mathcal{B}}_h$ (or ${\mathcal{P}}_h$), then $W_p(G) \leq W_p(G')$ and the equality holds if and only if $G \in {\mathcal{B}}_h$ (or $ G \in {\mathcal{P}}_h$).
\end{theorem}

\begin{proof}
By Theorem \ref{Wie_pol}, the maximum value of the Wiener polarity index is attained by a graph $G$ which attains also the maximum value with respect to the number $s(G) + b(G)$. 

By the construction, the graphs in ${\mathcal{B}}_h \cup{\mathcal{P}}_h$ obviously have the maximum number of branched hexagons among all the catacondensed benzenoid graphs or phenylenes with $h$ hexagons. If $h$ is even, these graphs are the only graphs attaining the upper bound with respect to the number of branched hexagons. If $h$ is odd, the upper bound with respect to the number of branched hexagons is attained also by some other graphs. To describe them, we first introduce Extension 3.

\textbf{Extension 3:} \textit{Let $G_0$ be a catacondensed benzenoid graph (or a phenylene) and let $h_0$ be an arbitrary terminal hexagon of $G_0$. We denote by $G$ a graph obtained from $G_0$ by attaching exactly one new hexagon to $h_0$ such that $h_0$ becomes a linear hexagon (and in the case of phenylenes, we also add a quadrilateral between $h_0$ and the new hexagon). See Figure \ref{dodas_ravno}.}
\begin{figure}[h!] 
\begin{center}
\includegraphics[scale=0.6]{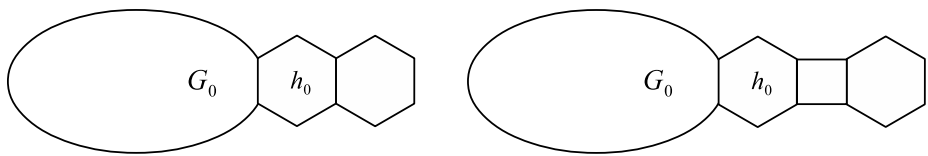}
\end{center}
\caption{\label{dodas_ravno} A catacondensed benzenoid graph (or a phenylene) $G$ obtained from $G_0$ by Extension 3.}
\end{figure}

If $h$ is an odd number, then ${\mathcal{B}}_h'$ (or ${\mathcal{P}}_h'$) is the set of all the graphs that can be obtained from the catacondensed benzenoid graph with exactly two hexagons (or from the phenylene with exactly two hexagons) by performing Extension 1 exactly $\left(\frac{h-1}{2}-1\right)$ times and performing Extension 3 exactly once (in an arbitrary order).
\noindent
It is clear that for any odd $h$ the upper bound with respect to the number of branched hexagons is attained exactly by the graphs in ${\mathcal{B}}_h  \cup{\mathcal{B}}_h'$ (or ${\mathcal{P}}_h \cup{\mathcal{P}}_h'$).

Moreover, by Lemma \ref{st_segmentov} the graphs in ${\mathcal{B}}_h \cup{\mathcal{P}}_h$ also have the maximum possible number of segments. On the other hand, by applying the same reasoning as in the proof of Lemma \ref{st_segmentov}, we deduce that the graphs in ${\mathcal{B}}_{h}'$ (or ${\mathcal{P}}_{h}'$) do not attain the upper bound with respect to the number of segments since the number of segments does not increase when Extension 3 is performed. Therefore, it is clear that $G$ attains the maximum value with respect to the Wiener polarity index if and only if $G \in {\mathcal{B}}_h$ (or $ G \in {\mathcal{P}}_h$). \qed 
\end{proof}

\section{Concluding remarks}

In the paper we have generalized a formula for calculating the Wiener polarity index of a graph. The obtained result can be used for many (molecular) graphs since the conditions in the main theorem do not prohibit $4$-cycles in a considered graph. As an example of the main result, the closed formulas for the Wiener polarity index are calculated for phenylenes and recalculated for catacondensed benzenoid graphs. Moreover, the closed formulas are used to characterize the graphs with the minimum and the maximum Wiener polarity index. Regarding the future work, it would be interesting to generalize the main result such that it could be used even for graphs sufficing weaker conditions. However, in such a case the formula would probably become much more complicated.

\section*{Acknowledgment} 

\noindent The author was financially supported by the Slovenian Research Agency (research core funding No. P1-0297 and J1-9109).


\end{document}